\documentclass{amsart}

\usepackage{amscd}
\usepackage{graphicx}
\usepackage{tikz}
\usepackage{overpic}

\newtheorem{theorem}{Theorem}
\newtheorem{proposition}{Proposition}

\newtheorem{corollary}{Corollary}

\theoremstyle{plain}
  \newtheorem{ques}{Question}
  
\theoremstyle{definition}

\theoremstyle{remark}

\begin{document}

\title[Pure cactus groups and configuration spaces of points]{Pure cactus groups and configuration spaces of points on the circle}

\author{Takatoshi Hama}
\address{Graduate School of Integrated Basic Sciences, Nihon University, 3-25-40 Sakurajosui, Setagaya-ku, Tokyo 156-8550, Japan}
\email{chta23018@g.nihon-u.ac.jp}

\author{Kazuhiro Ichihara}
\address{Department of Mathematics, 
College of Humanities and Sciences, Nihon University, 3-25-40 Sakurajosui, Setagaya-ku, Tokyo 156-8550, Japan}
\email{ichihara.kazuhiro@nihon-u.ac.jp}

\date{\today}



\begin{abstract}
In this article, we provide a summary of the results presented in the previous two papers of the authors and in the second author’s Master thesis, which concern pure cactus groups and configuration spaces of points on the circle.
\end{abstract}

\maketitle

%\section*{\S1. Introduction}
\section{Introduction}

The aim of this article is to provide a summary of the results presented in \cite{HamaIchiharaPJ3, hama2025presentationpurecactusgroup} and in the second author's thesis \cite{HamaMsthesis}, which concern pure cactus groups and configuration spaces of points on the circle.

It is known that the braid group acts naturally on multiple tensor products in a braided monoidal category.
As an analogue of this action, Henriques and Kamnitzer introduced the so-called \textit{cactus group} in \cite{HENRIQUES-KAMNITZER}, which acts on coboundary categories associated with braided monoidal categories.

We note that their work was motivated by the study of crystal bases, introduced by Kashiwara in \cite{Kashiwara}.
Although the term “cactus group” was first coined in \cite{HENRIQUES-KAMNITZER}, the group itself had already been studied in earlier works, including \cite{DJS03} and \cite{devadoss2000tessellationsmodulispacesmosaic}.

More precisely, for any integer $n \geq 2$, the \textit{cactus group of degree $n$}, denoted by $J_n$, is defined by a presentation with generators $s_{p,q}$ for $1 \leq p < q \leq n$, subject to the following relations:
\begin{itemize}
\item $s_{p,q}^2 = e$ for all $1 \leq p < q \leq n$,
\item $s_{p,q}s_{m,r} = s_{m,r}s_{p,q}$ if $[p, q] \cap [m, r] = \emptyset$,
\item $s_{p,q}s_{m,r} = s_{p+q-r,p+q-m}s_{p,q}$ if $[m, r] \subset [p, q]$,
\end{itemize}
where $1 \leq m < r \leq n$, $e$ denotes the identity element, and $[p, q]$ is the set ${p, p+1, \dots, q}$ for integers $p < q$.

Similar to the braid group, elements of the cactus group can be represented by planar diagrams composed of vertical strands.
Examples of such diagrams for $J_4$ are shown in Figure~\ref{Fig:element}.

\begin{figure}[htb]
\centering
    \begin{picture}(340,70)(0,0)
    \put(-27,0){
    \begin{tikzpicture}

\draw[ultra thick] (0,0) .. controls +(0,1) and +(0,-1) ..
(.5,2) ;
\draw[ultra thick] (.5,0) .. controls +(0,1) and +(0,-1) ..
(0,2) ;
\draw[ultra thick] (1,0) .. controls +(0,0) .. (1,2);
\draw[ultra thick] (1.5,0) .. controls +(0,0) .. (1.5,2);

\draw[ultra thick] (3,0) .. controls +(0,0) .. (3,2);
\draw[ultra thick] (3.5,0) .. controls +(0,1) and +(0,-1) ..
(4,2) ;
\draw[ultra thick] (4,0) .. controls +(0,1) and +(0,-1) ..
(3.5,2) ;
\draw[ultra thick] (4.5,0) .. controls +(0,0) .. (4.5,2);

\draw[ultra thick] (6,0) .. controls +(0,0) .. (6,2);
\draw[ultra thick] (6.5,0) .. controls +(0,0) .. (6.5,2);
\draw[ultra thick] (7,0) .. controls +(0,1) and +(0,-1) ..
(7.5,2) ;
\draw[ultra thick] (7.5,0) .. controls +(0,1) and +(0,-1) ..
(7,2) ;

\draw[ultra thick] (9,0) .. controls +(0,1) and +(0,-1) ..
(10,2) ;
\draw[ultra thick] (9.5,0) .. controls +(0,0) .. (9.5,2);
\draw[ultra thick] (10,0) .. controls +(0,1) and +(0,-1) ..
(9,2) ;
\draw[ultra thick] (10.5,0) .. controls +(0,0) .. (10.5,2);

\draw[ultra thick] (12,0) .. controls +(0,1) and +(0,-1) ..
(13.5,2) ;
\draw[ultra thick] (12.5,0) .. controls +(0,1) and +(0,-1) ..
(13,2) ;
\draw[ultra thick] (13,0) .. controls +(0,1) and +(0,-1) ..
(12.5,2) ;
\draw[ultra thick] (13.5,0) .. controls +(0,1) and +(0,-1) ..
(12,2) ;
\end{tikzpicture}
}
    \put(340,-10){$s_{14}$}
    \put(255,-10){$s_{13}$}
    \put(0,-10){$s_{12}$}
    \put(85,-10){$s_{23}$}
    \put(170,-10){$s_{34}$}

\end{picture}

\vspace{1.5cm}

   \begin{picture}(350,80)(0,0)
    \put(20,0){
\begin{tikzpicture}
\draw[ultra thick] (0,0) .. controls +(0,1) and +(0,-1) ..
(1,1.5) .. controls +(0,1) and +(0,-1) .. (0,3);
\draw[ultra thick] (1,0) .. controls +(0,1) and +(0,-1) ..
(0,1.5) .. controls +(0,1) and +(0,-1) .. (1,3);
\draw[ultra thick] (.5,0) .. controls +(0,0) .. (.5,3);
\draw[ultra thick] (1.4,0) .. controls +(0,0) .. (1.4,3);

\draw[ultra thick] (2.5,0) .. controls +(0,0) .. (2.5,3);
\draw[ultra thick] (3,0) .. controls +(0,0) .. (3,3);
\draw[ultra thick] (3.5,0) .. controls +(0,0) .. (3.5,3);
\draw[ultra thick] (4,0) .. controls +(0,0) .. (4,3);

\draw[ultra thick] (6,0) .. controls +(0,1) and +(0,-1) ..
(6.5,1.5) .. controls +(0,1) and +(0,-1) .. (7,3);
\draw[ultra thick] (6.5,0) .. controls +(0,1) and +(0,-1) ..
(6,1.5) .. controls +(0,1) and +(0,-1) .. (7.5,3);
\draw[ultra thick] (7,0) .. controls +(0,0) ..
(7,1.5) .. controls +(0,1) and +(0,-1) .. (6.5,3);
\draw[ultra thick] (7.5,0) .. controls +(0,0) ..
(7.5,1.5) .. controls +(0,1) and +(0,-1) .. (6,3);

\draw[ultra thick] (9,0) .. controls +(0,1) and +(0,-1) ..
(10.5,1.5) .. controls +(0,1) and +(0,-1) .. (10,3);
\draw[ultra thick] (9.5,0) .. controls +(0,1) and +(0,-1) ..
(10,1.5) .. controls +(0,1) and +(0,-1) .. (10.5,3);
\draw[ultra thick] (10,0) .. controls +(0,1) and +(0,-1) ..
(9.5,1.5) .. controls +(0,0) .. (9.5,3);
\draw[ultra thick] (10.5,0) .. controls +(0,1) and +(0,-1) ..
(9,1.5) .. controls +(0,0) .. (9,3);

\end{tikzpicture}
}
    \put(77,40){\large $=$}
    \put(55,-15){$ ( s_{13} )^2 = e$}
    
    \put(255,40){\large $=$}
    \put(240, -15){$s_{14}s_{12} = s_{34} s_{14}$}
    % \put(9.2,1.2){$s_{12}s_{13}$}
\end{picture}

\vspace{.5cm}

\caption{Diagrams for some elements of $J_4$}\label{Fig:element}
\end{figure}
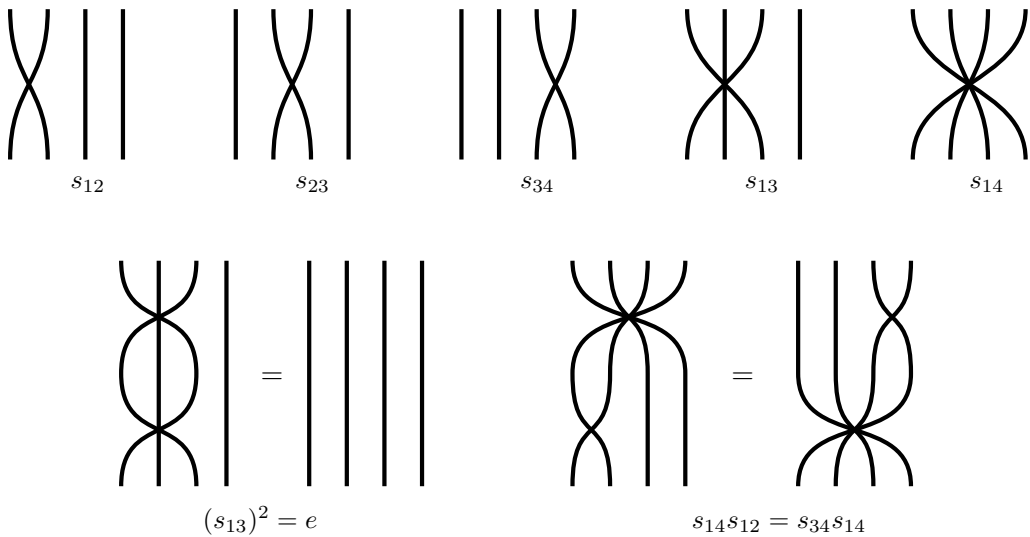

Owing to this diagrammatic representation, the cactus group $J_n$ admits a natural projection $\pi : J_n \to S_n$ onto the symmetric group $S_n$ of degree $n$.
The kernel of this projection is called the \textit{pure cactus group} of degree $n$, denoted by $PJ_n$.
For further details, see \cite[Subsection 3.1]{HENRIQUES-KAMNITZER} or \cite[Section 1]{genevois2022cactusgroupsviewpointgeometric}.

Also, in \cite[Theorem 9]{HENRIQUES-KAMNITZER}, it is shown that the pure cactus group $PJ_n$ is isomorphic to the fundamental group of the Deligne--Mumford compactification $\overline{M_{0,n+1}(\mathbb{R})}$
of the moduli space of real genus-zero curves with $n+1$ marked points.
This space is closely related to the configuration space of $n+1$ points on the circle.

In this article, we report our study the relationship between the pure cactus group and the compactification of the configuration space of points on the circle, focusing on the cases where the degree is three or four.

Throughout this paper, we denote the generators of $J_n$ using the notation $s_{23}$ (without a comma) instead of $s_{2,3}$, for brevity.

%\section*{\S2. Configuration spaces of points on $S^1$}
\section{Configuration spaces of points on $S^1$}

As mentioned above, by \cite[Theorem 9]{HENRIQUES-KAMNITZER}, the pure cactus group $PJ_n$ is isomorphic to the fundamental group of the Deligne--Mumford compactification $\overline{M_{0,n+1}(\mathbb{R})}$
of the moduli space of real genus-zero curves with $n+1$ marked points.

For $k \geq 3$, it is natural to regard $M_{0,k}(\mathbb{R})$ as the configuration space of $k$ distinct points on the circle $S^1$.
In this paper, we denote this space by $X(k)$, which can be explicitly described as:
\begin{align*}
X(k) = \mathrm{PGL}(2) \backslash \left\{ (\mathbb{P}^1)^k - \Delta \right\},
\end{align*}
where $\Delta = {(x_1, \dots, x_k) \mid x_i = x_j\ \text{for some} \ i \neq j}$, and the projective general linear group $\mathrm{PGL}(2)$ acts diagonally and freely.
That is, there exists a homeomorphism between $M_{0,k}(\mathbb{R})$ and $X(k)$.

On the other hand, a purely combinatorial compactification $\overline{X(k)}$ was introduced by M.~Yoshida in \cite{MYos96}.
See also \cite{KNY99} and \cite{MN00} for related work.

At present, it is not known whether $\overline{X(k)}$ is homeomorphic to $\overline{M_{0,k}(\mathbb{R})}$ in general.
However, for $k = 3, 4, 5$, the following are known:
\begin{itemize}
\item Both $\overline{X(3)}$ and $\overline{M_{0,3}(\mathbb{R})}$ consist of a single point.
\item Both $\overline{X(4)}$ and $\overline{M_{0,4}(\mathbb{R})}$ are homeomorphic to $S^1$.
\item Both $\overline{X(5)}$ and $\overline{M_{0,5}(\mathbb{R})}$ are homeomorphic to the closed non-orientable surface with Euler characteristic $-3$, namely, the connected sum of five projective planes.
\end{itemize}
See \cite{MYos96, AperyYoshida, EHKR}, for example.
This implies that, for $n \le 4$, the pure cactus group $PJ_n$ is isomorphic to the fundamental group of $\overline{X(n+1)}$, using $\overline{M_{0,n+1}(\mathbb{R})}$ as an intermediate space.
Thus, the following question arises naturally:

\medskip

\noindent
\textbf{Question.}
Can one show that $PJ_n$ is isomorphic to $\pi_1(\overline{X(n+1)})$ without using $\overline{M_{0,n+1}(\mathbb{R})}$\;?

\medskip

The following theorem provides an affirmative answer in the case $n = 3$.

\begin{theorem}[{\cite{HamaIchiharaPJ3}}]
Let $\widetilde{X(4)}$ be the universal cover of the space $\overline{X(4)}$, endowed with the action $\widetilde{\Gamma}$ of the fundamental group $\pi_1 (\overline{X(4)} )$ of $\overline{X(4)}$ as deck transformations. 
Let ${C_3}^{\{ 2\}}$ be the Cayley complex of the subgroup $J_3^{\{2\}}$ of the cactus group $J_3$. 
Then there exists an action $\Gamma$ of $PJ_3$ on ${C_3}^{\{2\}}$ and a bijective equivariant map $\varphi$ from $\widetilde{X(4)}$ to ${C_3}^{\{2\}}$ with respect to the actions of $\widetilde{\Gamma}$ and $\Gamma$, 
i.e., for any $g \in PJ_3$, there exists $\tilde{g} \in \pi_1 (\overline{X(4)} )$ such that the following diagram commutes;

\[
\begin{CD}
\widetilde{X(4)} @>{\widetilde{\Gamma}_{\tilde{g}}}>>\widetilde{X(4)}\\
\varphi@VVV   @VVV\varphi \\
{C_3}^{\{2\}} @>{\Gamma_g}>> {C_3}^{\{2\}}
\end{CD}
\]
\end{theorem}

The theorem above implies the following immediately. 

\begin{corollary}
The pure cactus group $PJ_3$ of degree three is isomorphic to the fundamental group $\pi_1 (\overline{X(4)} )$ of the compactification $\overline{X(4)}$ of the configuration space of four points on the circle. 
\end{corollary}

Here, we recall the definition of the Cayley complex.
Let $G$ be a group and $S$ a generating set of $G$.
The \textit{Cayley graph of $G$ with respect to the generating set $S$} is the graph whose vertex set is $G$, and whose edge set is $\{ {g, gs} \mid g \in G, s \in S \}$.
The cell complex obtained by attaching $2$-cells to each cycle formed by the relations of $G$ is called the \textit{Cayley complex} of $G$ with respect to $S$.

We now introduce the subgroup of $J_3$ used in the theorem above.
In general, for each integer $n \geq 2$ and subset $S \subset [2, n]$, let $J_n^S$ be the subgroup of $J_n$ generated by the elements $s_{p,q}$ for $1 \leq p < q \leq n$ such that $q - p + 1 \in S$, and defined by the following relations:
\begin{itemize}
\item $s_{p,q}^2 = e$ for every $1 \leq p < q \leq n$ satisfying $q - p + 1 \in S$,
\item $s_{p,q}s_{m,r} = s_{m,r}s_{p,q}$ for every $1 \leq p < q \leq n$ and $1 \leq m < r \leq n$ satisfying $[p, q] \cap [m, r] = \emptyset$ and $q - p + 1 \in S$,
\item $s_{p,q}s_{m,r} = s_{p+q-r, p+q-m}s_{p,q}$ for every $1 \leq p < q \leq n$ and $1 \leq m < r \leq n$ satisfying $[m, r] \subset [p, q]$ and $q - p + 1 \in S$.
\end{itemize}
For more details, see \cite[Section 5]{genevois2022cactusgroupsviewpointgeometric}.
Under this setting, the subgroup $J_3^{\{2\}}$ of $J_3$ is defined for $S = \{ 2 \} \subset \{ 2, 3 \}$.
It is the subgroup of $J_3$ generated by $s_{12}$ and $s_{23}$ (excluding $s_{13}$), with the presentation:
\[
\left\langle 
s_{12}, s_{23} 
\left| 
s_{12}^2= s_{23}^2= e
\right.
\right\rangle.
\]
It follows that $J_3^{\{2\}}$ is isomorphic to $\mathbb{Z}_2 * \mathbb{Z}_2$.

For the proof of the theorem above, please refer to \cite{HamaIchiharaPJ3}.
The key is the action of $PJ_3$ on the Cayley complex ${C_3}^{\{2\}}$, which is homeomorphic to the real line $\mathbb{R}$.

Actually, this action on ${C_n}^{[2, n-1]}$ for $PJ_n$ is essentially obtained in \cite{genevois2022cactusgroupsviewpointgeometric}.

\begin{proposition}\label{prop}
Let $\left( {C_n}^{[2, n-1]} \right)^{(0)}$ denote the 0-skeleton of the Cayley complex ${C_n}^{[ 2, n-1 ]}$, which is identified with $J_n^{[2,n-1]}$.
The map
\begin{align*}
\Gamma_0: PJ_n \times \left( {C_n}^{[ 2, n-1 ]} \right)^{(0)} &\longrightarrow \left( {C_n}^{[ 2, n-1 ]} \right)^{(0)}
\end{align*}
defined by
\begin{align*}
(g, h) \mapsto (\Gamma_{0})_{g}(h) :=
\begin{cases}
g h & \text{if\ } g h \in {J_n}^{[ 2, n-1 ]}, \\
g h s_{1,n} & \text{if\ } gh \not\in {J_n}^{[ 2, n-1 ]}
\end{cases}
\end{align*}
induces a group action $\Gamma$ of $PJ_n$ on ${C_n}^{[ 2, n-1 ]}$.
\end{proposition}

% \section*{\S3. A presentation of $PJ_4$}
\section{A presentation of $PJ_4$}

Hereafter, we focus on the cactus group $J_4$ of degree four.
Recall that $J_4$ has the following presentation:
$$
\left\langle 
\begin{array}{l}
s_{12}, s_{23}, s_{34}, \\
s_{13}, s_{24}, s_{14} 
\end{array}
\left| 
\begin{array}{l}
s_{12}^2= s_{23}^2= s_{34}^2= s_{13}^2= s_{24}^2= s_{14}^2 = e,\\ 
s_{12} s_{34} = s_{34} s_{12}, \ 
s_{12} s_{13} = s_{13} s_{23}, \ 
s_{23} s_{24} = s_{24} s_{34}, \\
s_{12} s_{14} = s_{14} s_{34}, \ 
s_{23} s_{14} = s_{14} s_{23}, \ 
s_{13} s_{14} = s_{14} s_{24}
\end{array}
\right.
\right\rangle.
$$

As noted in the previous section, both $\overline{X(5)}$ and $\overline{M_{0,5}(\mathbb{R})}$ are homeomorphic to the closed non-orientable surface of Euler characteristic $-3$, which is the connected sum of five projective planes.
This implies that $PJ_4$ admits the following presentation:
\begin{equation}\label{surfpres}
\begin{array}{c}
\langle
\alpha_1 , \alpha_2 , \alpha_3 , \alpha_4 , \alpha_5 \mid
\alpha_1^2 \alpha_2^2 \alpha_3^2 \alpha_4^2 \alpha_5^2 = e
\rangle
\end{array}
\end{equation}

On the other hand, another explicit presentation of $PJ_4$ was obtained purely algebraically in \cite[Appendix A]{BCL24}, using the Reidemeister--Schreier method.
In fact, as \cite[Theorem~5.5]{BCL24} states, it was shown that $PJ_4$ has the following presentation:
\begin{equation}\label{BCLpres}
\langle \alpha, \beta, \gamma, \delta, \epsilon \mid \alpha \gamma \epsilon \beta \epsilon \alpha^{-1} \delta^{-1} \beta \gamma \delta^{-1} = e \rangle
\end{equation}

This presentation is a simple one-relator presentation with five generators, but it does not appear to be directly related to the presentation described above in the presentation~\eqref{surfpres}.
Indeed, as noted in \cite[Remark 5.6]{BCL24}, it seems difficult to express the generators $\alpha_k$ in terms of the standard generators $s_{i,j}$ of the full cactus group $J_4$.

In \cite{hama2025presentationpurecactusgroup}, we obtain another presentation of $PJ_4$.

\begin{theorem}\label{thm:PJ4pres}
The pure cactus group $PJ_4$ admits the following presentation:
\begin{align*}
\left\langle 
g_1, \dots, g_{10} 
\left|
\begin{array}{l}
g_1g_{10}^{-1}g_2^{-1} = 
g_9g_5^{-1}g_4 = 
g_5g_1g_6^{-1} \\
= g_8g_{10}g_7^{-1} = 
g_8g_3^{-1}g_4 = 
g_2g_9g_7^{-1}g_6g_3^{-1} = e
\end{array}
\right.
\right\rangle
\end{align*}
This presentation can be transformed into the following:
\begin{align*}
\langle g_2, g_4, g_8, g_9, g_{10} \mid
g_2g_9g_{10}^{-1}g_8^{-1}g_4g_9g_2g_{10}g_8^{-1}g_4^{-1}=e
\rangle
\end{align*}
\end{theorem}

We remark that the generators $g_i$ given above are explicitly described in terms of the standard generating set $\{s_{ij}\}$ of the cactus group $J_4$ and are of minimal word length with respect to $\{s_{ij}\}$. 
That is, each $g_i$ has word length 4 or 5, and there are no elements in $PJ_4$ whose word length is less than or equal to 3.

We can confirm that the presentation in Theorem~\ref{thm:PJ4pres} is equivalent to the presentation~\eqref{surfpres},
as well as to the presentation~\eqref{BCLpres} given in \cite{BCL24}.
See \cite{hama2025presentationpurecactusgroup} for details.

The key to our proof of Theorem~\ref{thm:PJ4pres} is to consider the action of $PJ_4$ on the Cayley complex ${C_4}^{[2, 3]}$ of the subgroup $J_4^{[2,3]}$, as given in Proposition~\ref{prop}.
In fact, the subgroup $J_4^{[2,3]}$ is generated by $s_{12}, s_{23}, s_{34}, s_{13}, s_{24}$ (excluding $s_{14}$), and has the following presentation:
\[
\left\langle 
s_{12}, s_{23}, s_{34}, s_{13}, s_{24} 
\left| 
\begin{array}{l}
s_{12}^2= s_{23}^2= s_{34}^2= s_{13}^2= s_{24}^2=e,\\ 
s_{12} s_{34} = s_{34} s_{12}, \ 
s_{12} s_{13} = s_{13} s_{23}, \ 
s_{23} s_{24} = s_{24} s_{34}
\end{array}
\right.
\right\rangle.
\]
It follows that the Cayley complex ${C_4}^{[2, 3]}$ is isometric to the hyperbolic plane $\mathbb{H}^2$ up to scaling.
See Figure~\ref{word_length_2} for a local picture of the complex, and compare it with Figure~\ref{fig:45tesse}.

\begin{figure}[htb]
\centering
\begin{minipage}{.47\textwidth}
    \centering
    \vspace{-1.3cm}
\begin{overpic}[width=\textwidth]{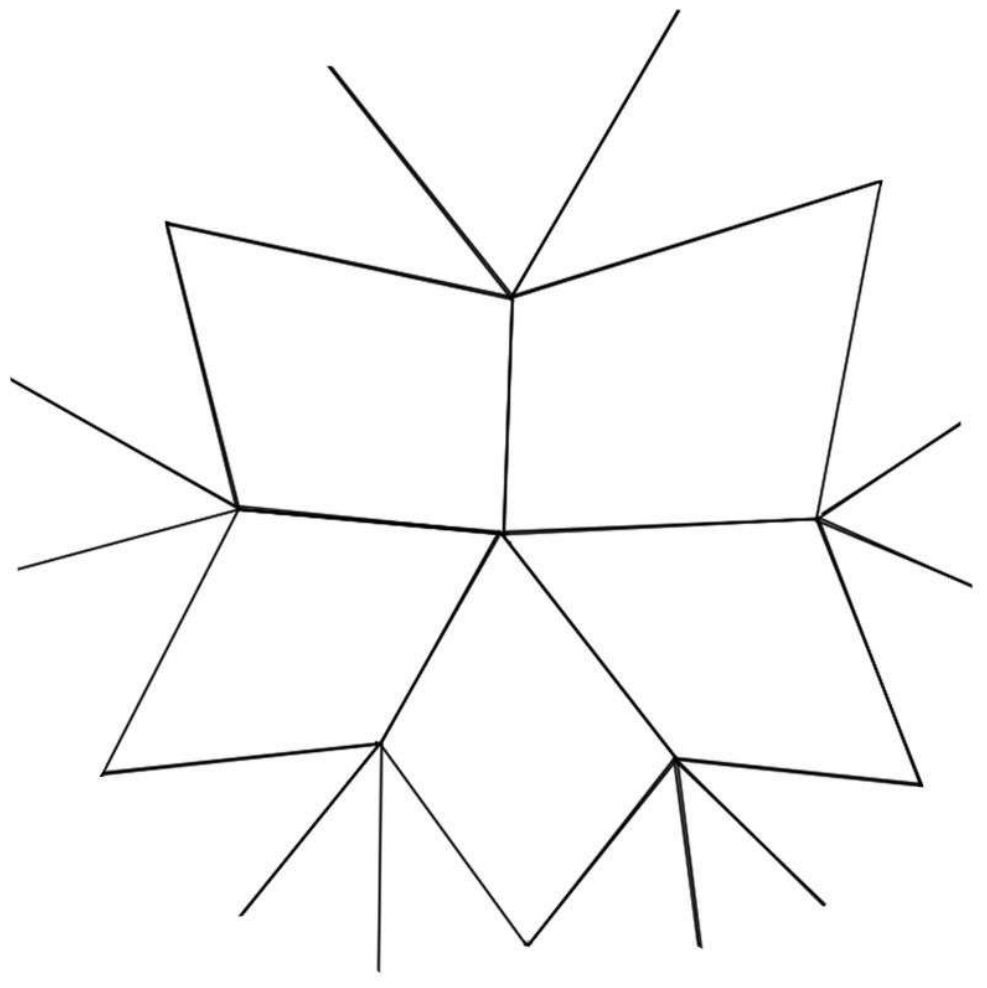}
\put(33,51){$e$} 

    \put(49,52){$s_{12}$}
    \put(31,61){$s_{13}$}
    \put(20,53){$s_{23}$}
    \put(24,39){$s_{24}$} 
    \put(44,38){$s_{34}$}

        \put(61,57){$s_{12}s_{23}$}
        \put(55,71){$s_{13}s_{23}$}
        \put(45,81){$s_{13}s_{24}$}
        \put(25,78){$s_{13}s_{34}$}
        \put(11,69){$s_{13}s_{12}$}
        \put(2,60){$s_{23}s_{12}$}
        \put(1,49){$s_{23}s_{34}$}
        \put(8,34){$s_{24}s_{34}$}
        \put(15,26){$s_{24}s_{12}$}
        \put(24,22){$s_{24}s_{13}$}
        \put(35,24){$s_{24}s_{23}$}
        \put(45,24){$s_{34}s_{23}$}
        \put(53,27){$s_{34}s_{13}$}
        \put(59,36){$s_{34}s_{12}$}
        \put(62,47){$s_{12}s_{24}$}       
\end{overpic}
\vspace{-2cm}
    \caption{Neighborhood of $e$ in the Cayley complex ${C_4}^{[ 2, 3 ]}$}
    \label{word_length_2}
%\end{figure}
\end{minipage}
\qquad 
\begin{minipage}{.45\textwidth}
%\begin{figure}[htb]
\centering
\includegraphics[width=.9\textwidth]{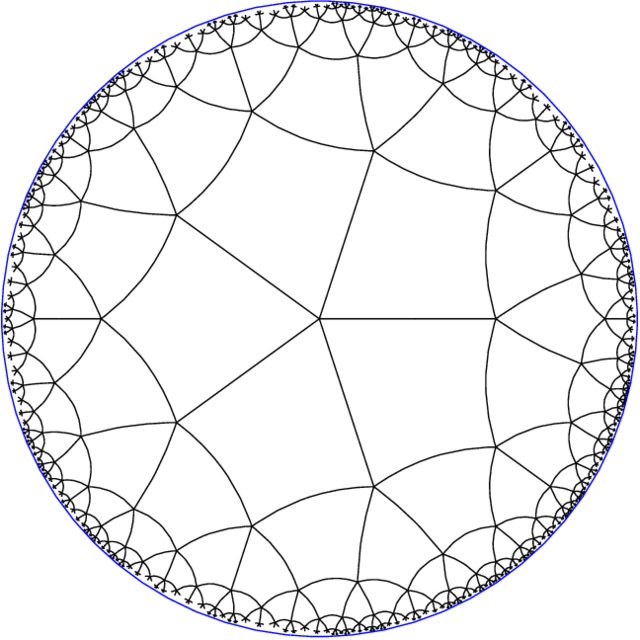}
\caption{The $\{4,5\}$-tesselation of the hyperbolic plane. (\cite[Figure 7]{BandeltChepoiEppstein})}
\label{fig:45tesse}
\end{minipage}
\end{figure}

Here we give an outline of the proof of Theorem~\ref{thm:PJ4pres}.

First, through brute-force enumeration, we listed all the elements of $PJ_4$ whose translation lengths are at most $4$ with respect to the action on ${J_4}^{[2, 3]}$.
As a result, we obtained the ten elements $g_1, \dots, g_{10}$ and their inverses in $PJ_4$ below. 
\[
\begin{array}{ll}
    g_{1} =  s_{13} s_{24} s_{12} s_{34} s_{14} \ , \qquad   &   
    g_{2} = s_{13}s_{24}s_{13}s_{24} \ ,          \\ 
    g_{3} = s_{13}s_{34}s_{23}s_{12} s_{14} \ ,    &
    g_{4} = s_{13}s_{34}s_{13}s_{23} s_{14}  \ ,   \\
    g_{5} = s_{23}s_{12}s_{23}s_{13} \ ,       &
    g_{6} = s_{23}s_{12}s_{24}s_{12}s_{14} \ ,   \\
    g_{7} = s_{23}s_{34}s_{13}s_{34} s_{14} \ ,   &
    g_{8} = s_{24}s_{34}s_{23}s_{34} \ ,  \\
    g_{9} = s_{24}s_{12}s_{23}s_{34}s_{14} \ ,  &
    g_{10} = s_{24}s_{23}s_{13}s_{34}s_{14}   
\end{array}
\]
In fact, if $d(e, g \cdot e) \le 4$ holds for some $g \in PJ_4$, then $g = g_i^{\pm1}$ for some $1 \leq i \leq 10$, and in fact, $d(e, g_i \cdot e) = 4$.

Using these elements, we can construct the Dirichlet polygon $D$ in $\mathbb{H}^2 \cong {C_4}^{[2, 3]}$ centered at $e$ for the action of $PJ_4$.
This polygon is illustrated in Figure~\ref{derichlet}.

{\small
\begin{figure}[htb]
    \centering
    %\hspace{-2.4cm}
    %\vspace{-1cm}
\begin{overpic}[width=\textwidth]{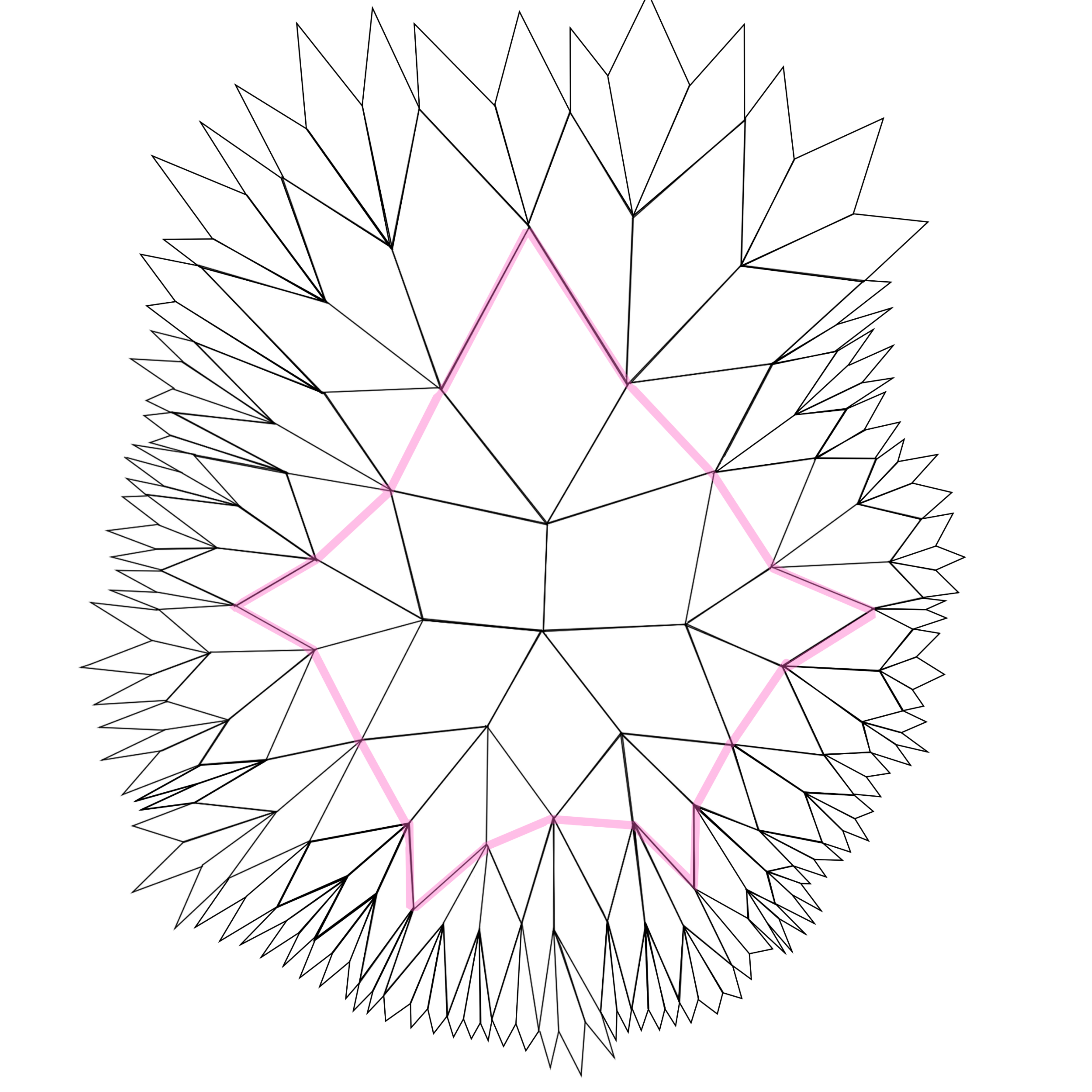}

\put(50,41){$e$} 

    \put(61,42){$s_{12}$}
    \put(50,50){$s_{13}$}
    \put(38,42){$s_{23}$}
    \put(45,32.5){$s_{24}$} 
    \put(57,32){$s_{34}$}

        \put(68,47){$s_{12}s_{23}$}
        \put(63.5,57){$s_{13}s_{23}$}
        \put(55.5,65){$s_{13}s_{24}$}
        \put(38,65){$s_{13}s_{34}$}
        \put(33,55.2){$s_{13}s_{12}$}
        \put(27,49){$s_{23}s_{12}$}
        \put(25,40.5){$s_{23}s_{34}$}
        \put(31.5,33){$s_{24}s_{34}$}
        \put(35,25){$s_{24}s_{12}$}
        \put(42,24){$s_{24}s_{13}$}
        \put(48.3,26){$s_{24}s_{23}$}
        \put(56,25){$s_{34}s_{23}$}
        \put(63,26){$s_{34}s_{13}$}
        \put(65,32){$s_{34}s_{12}$}
        \put(70,39){$s_{12}s_{24}$}
        
            \put(45,79){\small{$s_{13}s_{24}s_{23}$}}
            \put(20,45){\small{$s_{23}s_{34}s_{12}$}}
            \put(35,17){\small{$s_{24}s_{13}s_{23}$}}
            \put(62,19){\small{$s_{34}s_{13}s_{12}$}}
            \put(76,45){\small{$s_{12}s_{24}s_{34}$}}
\end{overpic}
    \caption{}
    \label{derichlet}
\end{figure}
}

The polygon $D$ has 20 sides.
Note that ten of these sides are not edges of the Cayley complex ${C_4}^{[2, 3]}$; rather, they are diagonals of certain 2-cells (quadrangles) in the complex.
Also, note that the interior angles around the vertices of $D$ are $\frac{2\pi}{5}$, $\frac{3\pi}{5}$, or $\frac{4\pi}{5}$.

The elements $g_1, \dots, g_{10}$ and their inverses define the side identifications of $D$ in $\mathbb{H}^2 \cong {C_4}^{[2, 3]}$.
Moreover, we can confirm that a complete hyperbolic surface is constructed after the identifications.
Then, by virtue of the famous Poincar\'{e} Polygon Theorem (cf. \cite{Poincare,Maskit}), it follows that $g_1, \dots, g_{10}$ generate $PJ_4$, and the compositions of generators
 $g_{3} g_6^{-1} g_{7} g_{9}^{-1} g_{2}^{-1}$, 
 $g_{3} g_{8}^{-1} g_{4}^{-1}$, 
 $g_{5} g_{9}^{-1} g_{4}^{-1}$, 
 $g_{5} g_{1} g_{6}^{-1}$, 
 $g_8 g_{10} g_7^{-1}$, 
 $g_{10} g_{1}^{-1} g_{2}$
form a complete set of relations.
Thus, we obtain the following presentation of $PJ_4$ as desired:
See \cite{hama2025presentationpurecactusgroup} for details.

Finally, we give an explanation of how to find the correspondence between the following two presentations:
\[
\begin{array}{ll}
\langle \alpha_1 , \alpha_2 , \alpha_3 , \alpha_4 , \alpha_5 \mid 
\alpha_1^2\alpha_2^2\alpha_3^2\alpha_4^2\alpha_5^2 =e \rangle , \\[12pt]
\left\langle 
g_1, \dots, g_{10} \left| 
\begin{array}{l}
g_{3} g_6^{-1} g_{7} g_{9}^{-1} g_{2}^{-1} = 
g_{3} g_{8}^{-1} g_{4}^{-1} = 
g_{5} g_{9}^{-1} g_{4}^{-1} \\
= g_{5} g_{1} g_{6}^{-1} = 
g_8 g_{10} g_7^{-1} = 
g_{10} g_{1}^{-1} g_{2} = e
\end{array}
\right.
\right\rangle 
\end{array}
\]
Viewing the fundamental polygon $D$, we see that the painted regions in Figure~\ref{Colored Mobius band} give five M\"{o}bius bands embedded in ${C_4}^{[ 2, 3 ]}/PJ_4$. 
After removing them, we can rebuild the remaining parts to obtain a sphere with five holes. 
See Figure~\ref{C_4^23/PJ_4}. 
Then we can confirm that ${C_4}^{[ 2, 3 ]} / PJ_4$ is homeomorphic to the connected sum $\sharp_5 \mathbb{RP}^2$ of five real projective planes. 
Setting the closed curve $\alpha_2$ in $\sharp_5 \mathbb{RP}^2$ shown in Figure~\ref{C_4^23/PJ_4}, we obtain a correspondence $\alpha_2 \longmapsto  g_8^{-1} g_4 g_9 g_2 g_{10} \in PJ_4$ by considering a lift of $\alpha_2$ in ${C_4}^{[ 2, 3 ]}$. 
In the same way, we can obtain the corresponding elements in $PJ_4$ for $\alpha_1$, $\alpha_3$, $\alpha_4$, $\alpha_5$. 

\begin{figure}[htb]
    \centering
    \begin{minipage}{.5\textwidth}
\begin{overpic}[width=\textwidth]{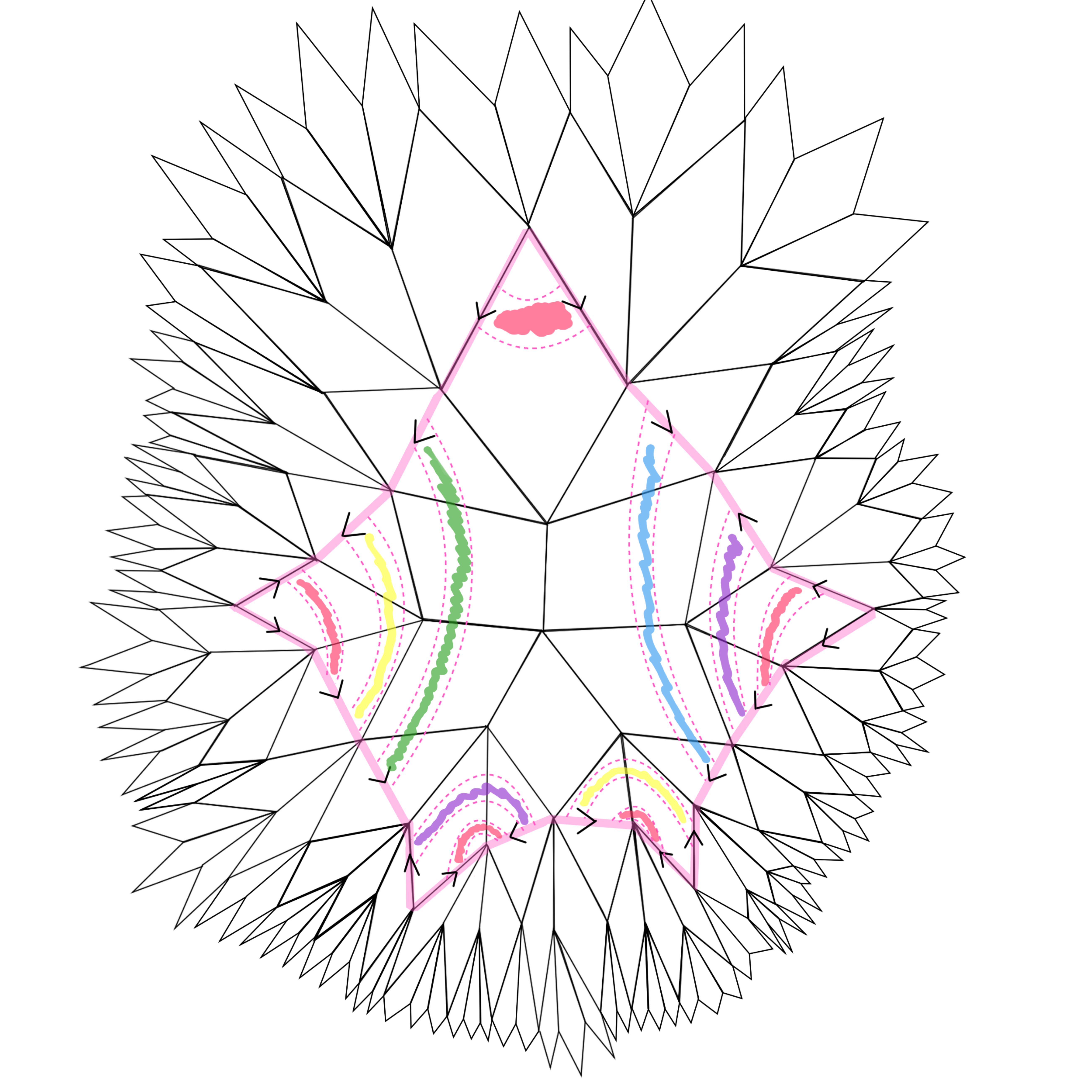}
\end{overpic}       
\caption{Five M\"{o}bius bands}
\label{Colored Mobius band}
    \end{minipage}
% \end{figure}
% \begin{figure}[htb]
%     \centering
    \begin{minipage}{.48\textwidth}
\begin{overpic}[width=\textwidth]{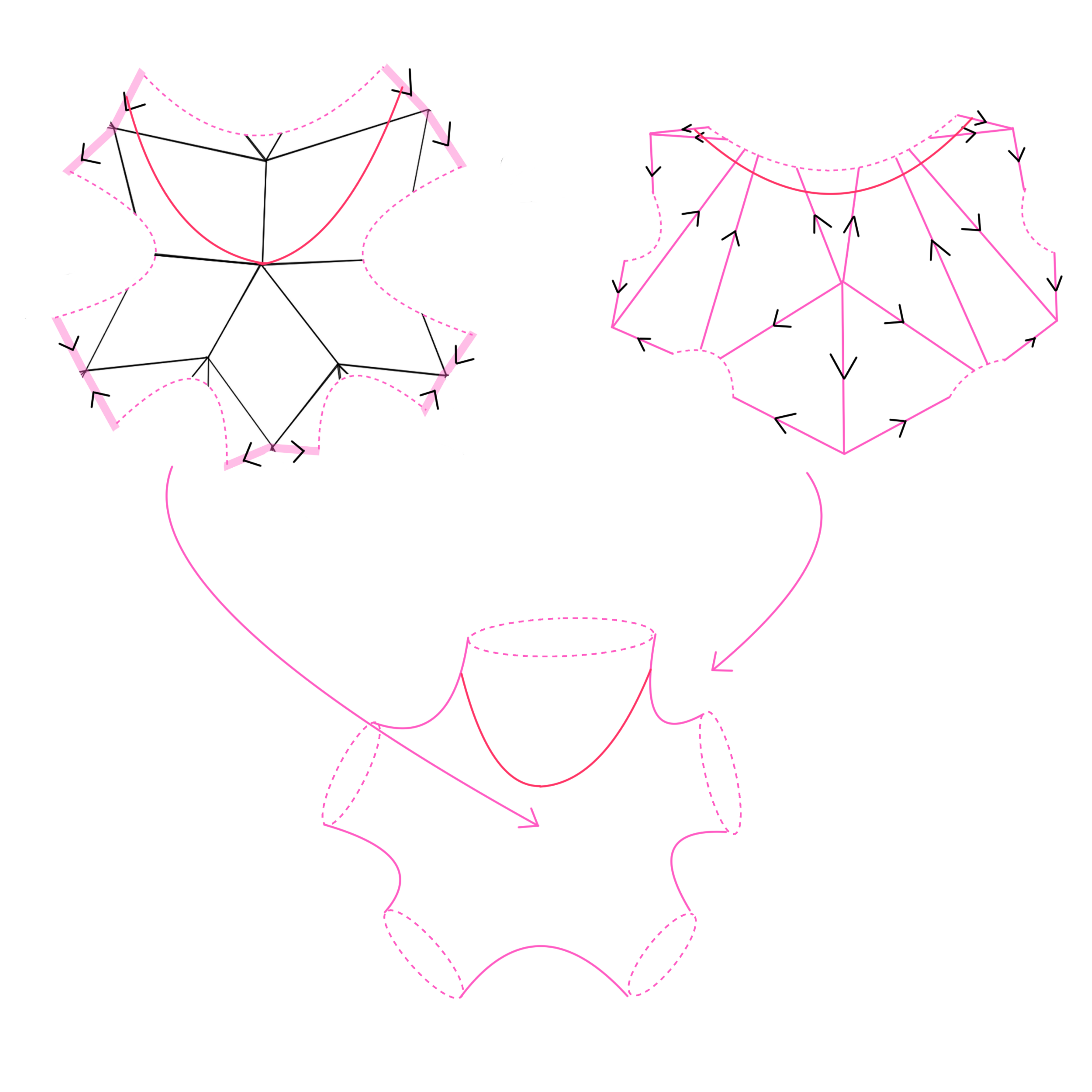}
\put(50,30){\textcolor{red}{$\alpha_2$}}
\end{overpic}
\caption{Finding the generator $\alpha_2$}
\label{C_4^23/PJ_4}
    \end{minipage}
\end{figure}

% \section*{\S4. $PJ_4$ and $X(5)$}
\section{$PJ_4$ and $\overline{X(5)}$}

As a consequence of the previous section, it can be proved that $PJ_4$ is isomorphic to the fundamental group of the connected sum of five projective planes, which is in turn isomorphic to $\pi_1(\overline{X(5)})$.
In this section, we provide an alternative proof by constructing a homeomorphism directly between $C_4^{[2,3]}/PJ_4$ and $\overline{X(5)}$. 

\begin{theorem}\label{PJ4 and X5}
Let $\overline{X(5)}$ denote the compactification of the configuration space $X(5)$ of five points on the circle.
Let ${C_4}^{[2, 3]}$ be the Cayley complex of the subgroup $J_4^{[2, 3]}$ of the cactus group $J_4$.
Then $\overline{X(5)}$ is homeomorphic to the quotient space of ${C_4}^{[2, 3]}$ under the action of the pure cactus group $PJ_4$ of degree four.
\end{theorem}

\begin{corollary}
The pure cactus group $PJ_4$ of degree four is isomorphic to the fundamental group $\pi_1(\overline{X(5)})$, where $\overline{X(5)}$ is the compactification of the configuration space of five points on the circle. \qed
\end{corollary}

We give a brief review of the cell complex structure of $\overline{X(5)}$ studied in \cite{AperyYoshida,MYos96}.
Each 2-cell of the complex $\overline{X(5)}$ corresponds to a connected component of $X(5)$, that is, a configuration of five distinct points, say 0, 1, 2, 3, 4, on $S^1$.
In the following, for example, the 2-cell corresponding to the configuration of points $0, 1, 2, 3, 4$ in clockwise order is coded by $[1, 2, 3, 4]$.
Note that $[1, 2, 3, 4]$ and $[4, 3, 2, 1]$ refer to the same 2-cell.
Then, there are twelve 2-cells in $\overline{X(5)}$: 
\[
\begin{array}{c}
\ [1234] , \quad [2134] , \quad [4123] , \quad [1324] , \quad [2341] , \quad [1243] ,   \\ 
\ [3124] , \quad [3412] , \quad [3241] , \quad [2314] , \quad [2431] , \quad [3142] .
\end{array}
\]
Two 2-cells are adjacent in $\overline{X(5)}$ if one is obtained from the other by switching a pair of adjacent points.
See Figure~\ref{adjacent_juzu} for an example.
Note that the 2-cells $[1234]$ and $[4123] = [3214]$ are adjacent; this corresponds to switching the points $0$ and $4$.

    \begin{figure}
        \centering
    \begin{overpic}[width=0.5\textwidth, angle=90]{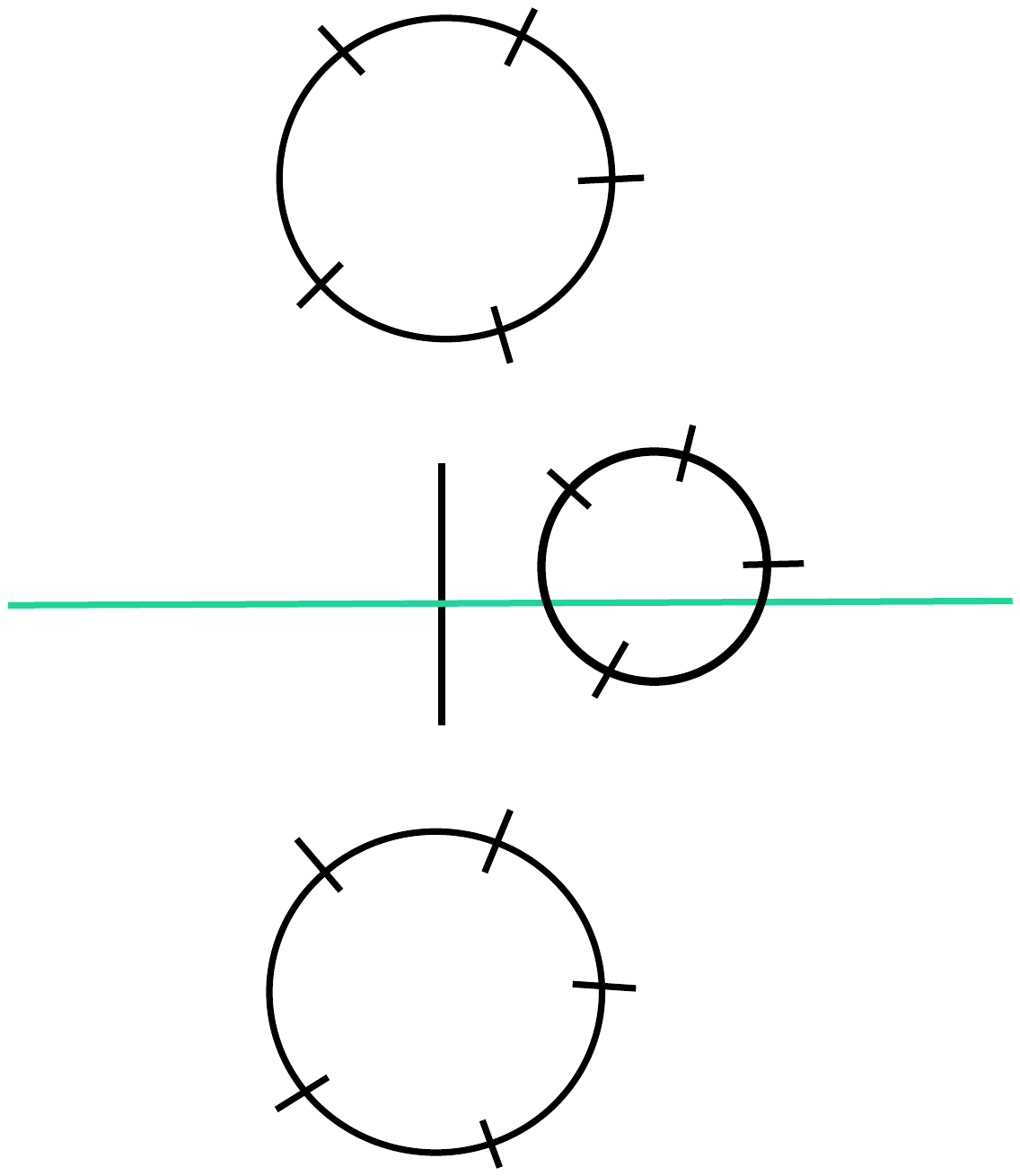}
%    \put(0,50){$\displaystyle X(5) \simeq \bigsqcup_{12} \overset{\circ}{D}$}   
%    \put(0,68){The compactification $\overline{X(5)}$ [c.f. Yoshida, 1996]}
        \put(25,45){$0$}
        \put(13,38){$4$}
        \put(36,38){$1$}
        \put(36,20){$2$}
        \put(13,20){$3$}

        \put(23,11){$[1234]$}
        \put(76,11){$[2134]$}

        \put(50,55){$0$}
        \put(57,38){\small{$1=2$}}
        \put(40,48){$4$}
        \put(42,37){$3$}

        \put(76,45){$0$}
        \put(90,35){$2$}
        \put(87,20){$1$}
        \put(66,20){$3$}
        \put(64,34){$4$}
    \end{overpic}
        \caption{Adjacent 2-cells; $[1234]$ and $[2134]$}\label{adjacent_juzu}
    \end{figure}

Then, the entire cell complex structure and its dual of $\overline{X(5)}$ are visualized in Figure~\ref{overline_X(5)_its_dual}.
In the figure, the 12 pentagons represent the 2-cells of $\overline{X(5)}$ (i.e., the connected components of $X(5)$).

\begin{figure}[htb]
\vspace{-1cm}
\centering
\begin{overpic}[width=.7\textwidth]{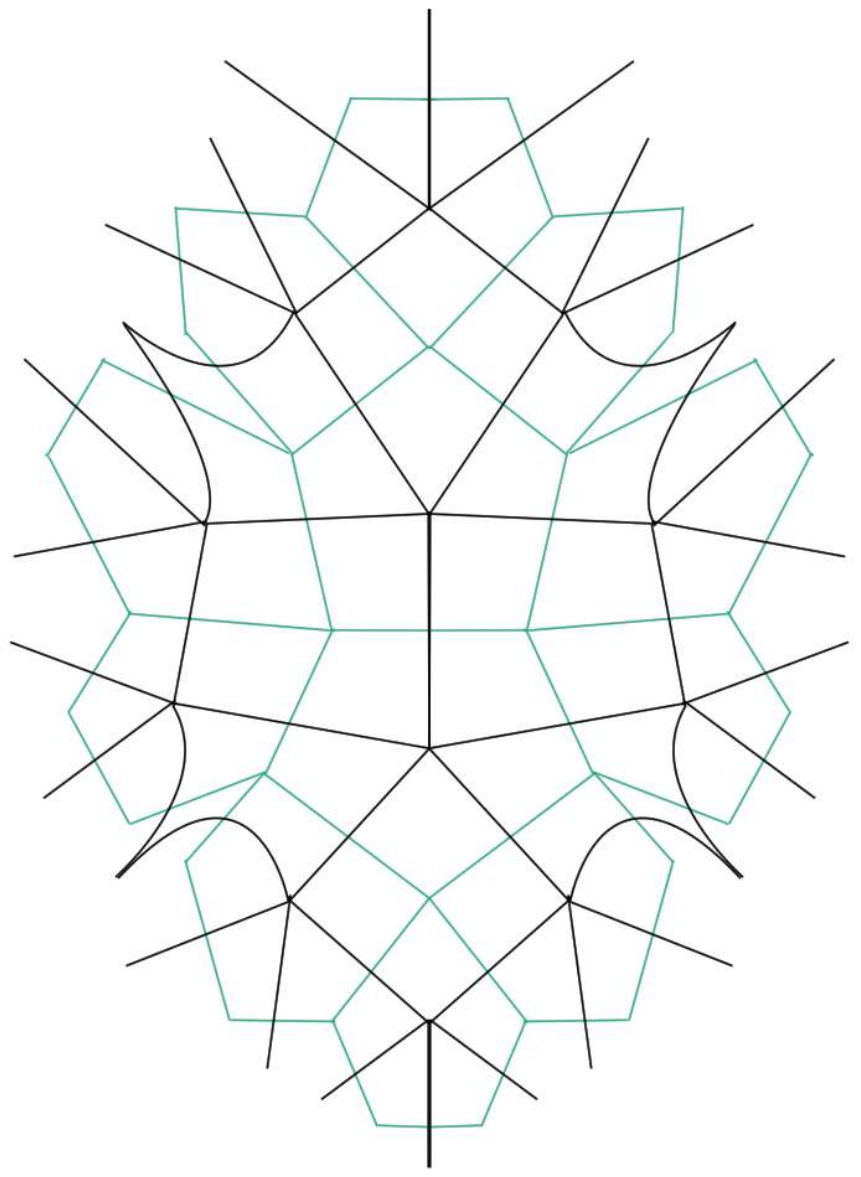}

        \put(22,83){$[3124]$}
        \put(33,87){$[1342]$}
        \put(45,83){$[2314]$}

        \put(20,78){$[1243]$}
        \put(45,78){$[1432]$}
        
        \put(14,74){$[1324]$}
        \put(52,74){$[2134]$}

        \put(15,68){$[1432]$}
        \put(51,68){$[1243]$}

        \put(9,65){$[3142]$}
        \put(57,65){$[3142]$}

        \put(6,54){$[2134]$}
        \put(60,54){$[1324]$}

        \put(6,48){$[3124]$}
        \put(60,48){$[2314]$}

        \put(8,40){$[1342]$}
        \put(58,40){$[1342]$}

        \put(14,34){$[3241]$}
        \put(54,34){$[3412]$}

        \put(12,30){$[2314]$}
        \put(54,30){$[3124]$}

        \put(22,24){$[3412]$}
        \put(45,24){$[3241]$}

        \put(24,21){$[2134]$}
        \put(41,21){$[1324]$}

        \put(33,18){$[3142]$}

        \put(33,76){$[3142]$} 

            \put(26,67){$[3241]$}
            \put(40,67){$[3412]$}

            \put(20,56){$[2314]$}
            \put(33,57){$[3214]$}
            \put(45,56){$[3124]$}

            \put(19,47){$[1324]$}
            \put(33,45){$[1234]$}
            \put(48,47){$[2134]$}

            \put(24,31){$[1432]$}
            \put(33,25){$[1342]$}
            \put(42,31){$[1243]$}
\end{overpic}
\vspace{-2cm}
\caption{Cell complex structure of $\overline{X(5)}$ and its dual.}
    \label{overline_X(5)_its_dual}
\end{figure}

In the following proof, we consider the dual cell complex structure of $\overline{X(5)}$. 
That is, the vertices (0-cells) are labeled by $[1234]$, $[2134]$, and so on. 

\begin{proof}[Proof of Theorem~\ref{PJ4 and X5}]
We consider the cell complex structure on the surface $F:=C_4^{[2,3]}/PJ_4$ induced naturally from $C_4^{[2,3]}$. 
See Figure~\ref{derichlet}. 
There are 15 2-cells (quadrangles) and 30 1-cells (edges).

Set the following correspondence $\varphi$ between 
$(\overline{X(5)})^0$ and $(C_4^{[2,3]}/PJ_4)^0$. 
\[
\begin{array}{clll}
%\varphi: & \overline{X(5)})^0 & \longrightarrow & (F)^0 \\
        &[1234] & \longmapsto & [e] \\
        &[2134] & \longmapsto & [s_{12}] \\
        &[4123] & \longmapsto & [s_{13}] \\
        &[1324] & \longmapsto & [s_{23}] \\
        &[2341] & \longmapsto & [s_{24}] \\
        &[1243] & \longmapsto & [s_{34}] \\
        &[3124] & \longmapsto & [s_{13}s_{23}] \\
        &[3412] & \longmapsto & [s_{13}s_{24}] \\
        &[3241] & \longmapsto & [s_{13}s_{34}] \\
        &[2314] & \longmapsto & [s_{13}s_{12}] \\
        &[2431] & \longmapsto & [s_{24}s_{23}] \\
        &[3142] & \longmapsto & [s_{13}s_{24}s_{23}] \\

\end{array}
\]

Let us show that this correspondence induces a well-defined, bijective cell map between $\overline{X(5)}$ and $F=C_4^{[2,3]}/PJ_4$. 

It is sufficient to prove that each set of vertices that span a cell in $\overline{X(5)}$ is mapped to a set of vertices that span a cell in $F$.

We list all of the $2$-cells in $\overline{X(5)}$:
\[
\begin{array}{cc}
     \langle [3124], [1324], [2314], [2134] \rangle, &
     \langle [3124], [3142], [1342], [1324] \rangle, \\
     \langle [3214], [3142], [1342], [1324] \rangle, &
     \langle [3412], [1243], [3124], [3214] \rangle, \\
     \langle [3412], [2134], [1234], [1243] \rangle, &
     \langle [3412], [1432], [1342], [2134] \rangle, \\
     \langle [3142], [4132], [1432], [3412] \rangle, &
     \langle [3142], [1342], [2134], [4132] \rangle, \\
     \langle [4231], [3241], [1243], [1342] \rangle, &
     \langle [1324], [1234], [1432], [3241] \rangle, \\
     \langle [1432], [3241], [3214], [2314] \rangle, &
     \langle [1342], [1243], [1234], [1432] \rangle, \\
     \langle [1234], [2134], [3124], [3214] \rangle, &
     \langle [1234], [1324], [2314], [3214] \rangle, \\
     \langle [3214], [3412], [3142], [3241] \rangle. 
\end{array}
\]

In the above list, for example, $\langle [3124], [1324], [2314], [2134] \rangle$ indicates a cycle of the vertices in this order.
See Figure~\ref{overline_X(5)_its_dual}.
For instance, let us consider the case of the $2$-cell $\langle [3124], [1324], [2314], [2134] \rangle$ in $\overline{X(5)}$.
The vertices $[3124]$, $[1324]$, $[2314]$, and $[2134]$ correspond to $[s{12}]$, $[s_{13}s_{12}]$, $[s_{23}]$, and $[s_{13}s_{23}]$ in $C_4^{[2,3]}$ by $\varphi$, respectively.

By the definition of the Cayley graph, we have:
$
\langle [s_{12}], [s_{13}s_{23}] \rangle 
=
\langle [s_{12}], [s_{12}s_{13}] \rangle ,
$ 
and 
$
\langle [s_{13}s_{12}], [s_{23}] \rangle
=
\langle [s_{23}s_{13}], [s_{23}] \rangle ,
$ 
which are edges in $C_4^{[2,3]}$.

The element $s_{13}s_{12}$ is mapped to $s_{12}s_{23}$ under the identification $g_5^{-1}$ as follows. 
\[
  g_5^{-1}  s_{13}s_{12}  
  = (s_{23}s_{12}s_{23}s_{13} )^{-1} s_{13}s_{12}  
  = s_{12}s_{23}s_{12}s_{13}s_{13}s_{12}  = s_{12}s_{23} 
\]
Therefore, $\langle [s_{12}], [s_{13}s_{12}] \rangle$ becomes a $1$-cell in $F$, since $\{ s_{12}, s_{12}s_{23} \}$ is an edge of the Cayley graph of $J_4^{[2,3]}$ and $[s_{13}s_{12}] = [s_{12}s_{23}]$ holds in $F$.
In the same way, since $g_5(s_{13}s_{23}) = s_{23}s_{12}$, the pair $\langle [s_{23}], [s_{13}s_{23}] \rangle$ also forms a $1$-cell in $F$.
It follows that the vertices $[s_{12}]$, $[s_{13}s_{12}]$, $[s_{23}]$, and $[s_{13}s_{23}]$ form a cycle in $F$.
By the definition of the Cayley complex, these vertices span a $2$-cell:
$\langle [s_{12}], [s_{13}s_{12}], [s_{23}], [s_{13}s_{23}] \rangle$ in $F$. 

The following calculations show that the other $2$-cells are similarly mapped to $2$-cells in $F$.
 
 \noindent
$\langle [3124], [3142], [1342], [1324] \rangle$: 
 \[
 \begin{array}{l}
      \varphi([3124]) = [s_{13}s_{23}] = [g_5^{-1} (s_{23}s_{12})] = [g_1 (s_{34}s_{13})]   \\ 
      \varphi([3142]) = [s_{13}s_{24}s_{23}] = [g_3(s_{34}s_{13}s_{12})] =[g_3 (g_6^{-1} (s_{23}s_{34}s_{12}))] \\ 
      \varphi([1342]) = [s_{24}s_{23}] = [g_8^{-1} (s_{23}s_{34})] \\ 
      \varphi([1324]) = [s_{23}]\\
\end{array}
\]
$\langle [3214], [3142], [1342], [1324] \rangle$: 
\[
\begin{array}{l}
      \varphi([3124]) = [s_{13}s_{23}] = [g_1 (s_{34}s_{13})] \\ 
      \varphi([1243]) = [s_{34}] \\ 
      \varphi([3214]) = [s_{13}s_{34}] = [g_3 (s_{34}s_{23})] \\ 
      \varphi([3142]) = [g_3 (s_{34}s_{13}s_{12})]
 \end{array}
 \]
$\langle [3412], [1243], [3124], [3214] \rangle$: 
\[
\begin{array}{l}
      \varphi([3412]) = [s_{13}s_{24}] = [g_1 (s_{34}s_{12})] \\ 
      \varphi([1243]) = [s_{34}] \\ 
      \varphi([3124]) = [s_{13}s_{23}] = [g_1 (s_{34}s_{13})] \\ 
      \varphi([3214]) = [s_{13})]
 \end{array}
 \]
$\langle [3412], [2134], [1234], [1243] \rangle$: 
 \[
\begin{array}{l}
      \varphi([3412]) = [s_{13}s_{24}] = [g_1 (s_{12}s_{34})] \\ 
      \varphi([2134]) = [s_{12}] \\ 
      \varphi([1234]) = [e] \\ 
      \varphi([1243]) = [s_{34}]
 \end{array}
 \]
$\langle [3412], [1432], [1342], [2134] \rangle$: 
 \[
\begin{array}{l}
      \varphi([3412]) = [s_{13}s_{24}] = [g_2 (s_{24}s_{13})] = [g_1 (s_{12}s_{34})] \\ 
      \varphi([1432]) = [s_{24}] \\ 
      \varphi([1342]) = [s_{24}s_{23}] = [g_{10} (s_{12}s_{24})]\\ 
      \varphi([2134]) = [s_{12}]
 \end{array}
 \]
 $ \langle [3142], [4132], [1432], [3412] \rangle$: 
 \[
\begin{array}{l}
      \varphi([3142]) = [s_{13}s_{24}s_{23}] = [g_2 (s_{24}s_{13}s_{23})] \\ 
      \varphi([4132]) = [s_{13}s_{12}] = [g_4^{-1} (s_{24}s_{12})] \\ 
      \varphi([1432]) = [s_{24}] = [g_{10} (s_{12}s_{24})]\\ 
      \varphi([3412]) = [s_{13}s_{24}] = [g_2 (s_{24}s_{13})]
 \end{array}
 \]
$\langle [3142], [1342], [2134], [4132] \rangle$: 
 \[
\begin{array}{l}
      \varphi([3142]) = [s_{13}s_{24}s_{23}] = [g_2 (s_{24}s_{13}s_{23})] = [g_2 (g_9 (s_{12}s_{24}s_{34}))] \\ 
      \varphi([1342]) = [s_{24}s_{23}] = [g_{10}(s_{12}s_{24})] \\ 
      \varphi([2134]) = [s_{12}] \\
      \varphi([4132]) = [s_{13}s_{12}] = [g_5 (s_{12}s_{23})] = [g_4 (s_{24}s_{12})] 
 \end{array}
 \]
$\langle [4231], [3241], [1243], [1342] \rangle$:
 \[
\begin{array}{l}
      \varphi([4231]) = [s_{23}] \\ 
      \varphi([3241]) = [s_{13}s_{34}] = [g_{4}(s_{23}s_{24})] =[g_3 (s_{34}s_{23})] \\ 
      \varphi([1243]) = [s_{34}] \\
      \varphi([1342]) = [s_{24}s_{23}] = [g_8^{-1} (s_{23}s_{34})] 
 \end{array}
 \]
$\langle [1324], [1234], [1432], [3241] \rangle$:
 \[
\begin{array}{l}
      \varphi([1324]) = [s_{23}] \\ 
      \varphi([1234]) = [e] \\ 
      \varphi([1432]) = [s_{34}] \\
      \varphi([3241]) = [s_{13}s_{34}] = [g_3 (s_{34}s_{23})] = [g_4 (s_{24}s_{34})] 
 \end{array}
 \]
 $\langle [1432], [3241], [3214], [2314] \rangle$:
 \[
\begin{array}{l}
      \varphi([1432]) = [s_{24}] \\ 
      \varphi([3241]) = [s_{13}s_{34}] = [g_4 (s_{24}s_{34})] \\ 
      \varphi([3214]) = [s_{13}] \\
      \varphi([2314]) = [s_{13}s_{12}] = [g_4 (s_{24}s_{12})] 
 \end{array}
 \]
$\langle [1342], [1243], [1234], [1432] \rangle$: 
 \[
\begin{array}{l}
      \varphi([1342]) = [s_{24}s_{23}] \\ 
      \varphi([1243]) = [s_{34}] \\ 
      \varphi([1234]) = [e] \\ 
      \varphi([1432]) = [s_{24}] 
 \end{array}
 \]
$\langle [1234], [2134], [3124], [3214] \rangle$: 
 \[
\begin{array}{l}
      \varphi([1234]) = [e] \\ 
      \varphi([2134]) = [s_{12}] \\ 
      \varphi([3124]) = [s_{13}s_{23}] \\ 
      \varphi([3214]) = [s_{13}]
\end{array}
 \]
$\langle [1234], [1324], [2314], [3214] \rangle$: 
 \[
\begin{array}{l}
      \varphi([1234]) = [e] \\ 
      \varphi([1324]) = [s_{23}] \\ 
      \varphi([2314]) = [s_{13}s_{12}] \\ 
      \varphi([3214]) = [s_{13}]
\end{array}
 \]
$\langle [3214], [3412], [3142], [3241] \rangle$: 
 \[
\begin{array}{l}
      \varphi([3214]) = [s_{13}] \\ 
      \varphi([3412]) = [s_{13}s_{24}] \\ 
      \varphi([3142]) = [s_{13}s_{34}] \\ 
      \varphi([3241]) = [s_{13}s_{24}s_{23}]
\end{array}
 \]

Thus, each $2$-cell in $\overline{X(5)}$ is mapped to a $2$-cell in $F$ under $\varphi$, and hence $\varphi$ defines a bijective, cellular map.
Consequently, $\varphi$ is a homeomorphism.
\end{proof}

\section{Questions}

In this section, we collect some open questions. 

It is shown in \cite[Proposition 6.1]{genevois2022cactusgroupsviewpointgeometric} 
that the cactus group $J_n$ is not hyperbolic for any $n \geq 6$.
For small values of $n$, the hyperbolicity of $J_n$ is well understood:
\begin{itemize}
\item $J_2 \cong \mathbb{Z}/2\mathbb{Z}$,
\item $J_3$ is virtually the infinite dihedral group 
which is isomorphic to $\mathbb{Z}/2\mathbb{Z} * \mathbb{Z}/2\mathbb{Z}$,
\item $J_4$ is virtually a hyperbolic surface group,
\end{itemize}
and all of these are hyperbolic.
(See, for instance, \cite{genevois2022cactusgroupsviewpointgeometric,BCL24}.)

Thus, it is natural to consider the following question.

\begin{ques}[{\cite[Question 8.7]{genevois2022cactusgroupsviewpointgeometric}}]
Are the cactus group $J_5$ and the pure cactus group $PJ_5$ hyperbolic?
\end{ques}

In \cite[Theorem 9]{HENRIQUES-KAMNITZER}, Henriques and Kamnitzer showed that 
$PJ_n$ is isomorphic to the fundamental group of 
the Deligne--Mumford compactification $\overline{M_{0,n+1}}(\mathbb{R})$ 
of the moduli space of real genus-zero curves with $n+1$ marked points.
Therefore, the following question is closely related to 
the geometric structure of $\overline{M_{0,6}}(\mathbb{R})$, 
which is known to be a non-orientable closed smooth 3-manifold. 

\begin{ques}
Is the 3-manifold whose fundamental group is isomorphic to $PJ_5$ 
a hyperbolic 3-manifold?
\end{ques}

On the other hand, since non-elementary 
({\it i.e.}, infinite and not virtually cyclic) hyperbolic groups 
have always exponential growth rate \cite{GhysdelaHarpe}, 
the following question also arises naturally in connection with 
the above question.

\begin{ques}
% Does the cactus group $J_5$ have exponential growth rate?
Can we determine the growth rate of the cactus group, in particular $J_5$?
\end{ques}

It is also known that, for any finitely generated hyperbolic group, 
the generating function encoding its growth function is rational 
(see, for instance, \cite{GhysdelaHarpe}).
This leads to another natural question:

\begin{ques}
Can we describe the generating function that represents 
the growth function of the cactus group $J_5$?
\end{ques}

\section*{Acknowledgments}
The authors would like to thank Takuya Sakasai for his valuable feedback and Carl-Fredrik Nyberg Brodda for inspiring future research directions. 
They also thank Tetsuya Ito for his comments on the questions in the last section.

\bibliographystyle{amsalpha}
%\bibliography{ref}

\providecommand{\bysame}{\leavevmode\hbox to3em{\hrulefill}\thinspace}
\providecommand{\MR}{\relax\ifhmode\unskip\space\fi MR }
\providecommand{\MRhref}[2]{%
  \href{http://www.ams.org/mathscinet-getitem?mr=#1}{#2}
}
\providecommand{\href}[2]{#2}

\end{document}